\documentclass{amsart}
\usepackage{latexsym,amssymb,amsmath,amsthm,amscd,graphicx}
\usepackage{esint} 
\usepackage{color}
\usepackage{comment}
\numberwithin{equation}{section}
\newtheorem{theorem}{Theorem}[section]
\newtheorem{lemma}[theorem]{Lemma}

\theoremstyle{definition}

\newtheorem{remark}[theorem]{Remark}

\theoremstyle{remark}

\begin{document}

\title{Global universality of the two-layer neural network
with the $k$-rectified linear unit}

\author{Naoya Hatano, Masahiro Ikeda, Isao Ishikawa, and Yoshihiro Sawano}
\address[Naoya Hatano]{Department of Mathematics, Faculty of Science and Technology, Keio University, 3-14-1 Hiyoshi, Kohoku-ku, Yokohama 223-8522, Japan/Center for Advanced Intelligence Project, RIKEN, Japan, and Department of Mathematics, Chuo University, 1-13-27, Kasuga, Bunkyo-ku, Tokyo 112-8551, Japan,}
\address[Masahiro Ikeda]{Center for Advanced Intelligence Project, RIKEN, Japan/Department of Mathematics, Faculty of Science and Technology, Keio University, 3-14-1 Hiyoshi, Kohoku-ku, Yokohama 223-8522, Japan
}
\address[Isao Ishikawa]{Center for Advanced Intelligence Project, RIKEN, Japan, and Department of Engineering for Production and Environment, Graduate School of Science and Engineering, Ehime University, 3 Bunkyo-cho, Matsuyama, Ehime 790-8577, Japan,}
\address[Yoshihiro Sawano]{Department of Mathematics, Faculty of Science and Technology, Keio University, 3-14-1 Hiyoshi, Kohoku-ku, Yokohama 223-8522, Japan/Center for Advanced Intelligence Project, RIKEN, Japan, and Department of Mathematics, Chuo University, 1-13-27, Kasuga, Bunkyo-ku, Tokyo 112-8551, Japan}
\email[Naoya Hatano]{n.hatano.chuo@gmail.com}
\email[Masahiro Ikeda]{masahiro.ikeda@riken.jp}
\email[Isao Ishikawa]{isao.ishikawa@riken.jp}
\email[Yoshihiro Sawano]{yoshihiro-sawano@celery.ocn.ne.jp}

\maketitle

\begin{abstract}
This paper
concerns
the universality of the two-layer neural network with the $k$-rectified linear unit 
activation function with $k=1,2,\ldots$ 
with
a suitable norm without any restriction on the shape of the domain.
This type of result is called global universality, which extends the previous result for $k=1$
by the present authors.
This paper covers $k$-sigmoidal functions
as
an
application of the fundamental result
on $k$-rectified linear unit  functions.
\end{abstract}

\section{Introduction}

The goal of this note is to 
specify the closure of linear subspaces generated by $k$-ReLU
($k$-ReLU for short)
functions under various norms.
As in
\cite{SiXu}, for $k\in{\mathbb N}$, we set
\begin{equation}\label{eq:210616-1}
{\rm ReLU}(x)
:=
\begin{cases}
0, & x\le0,\\
x & x>0,
\end{cases}
\quad
\text{ReLU}^k(x)
:=
{\rm ReLU}(x)^k.
\end{equation}
The function $k$-ReLU is called the $k$-rectified linear unit,
which is introduced to compensate for
the properties
that ${\rm ReLU}$ does not have.
Our approach will be
a
completely
mathematical one.
Recently, 
more and more attention has been paid
to the $k$-ReLU function
as well as the original ReLU function.
For example,
if $k \ge 2$,
the function $k$-ReLU is in the class $C^{k-1}$,
so that it is smoother than the ReLU function.
When we study neural networks,
the function $k$-ReLU is called
an activation function.
As in \cite{GSPV19},
$k$-ReLU functions
are used 
to reduce
the amount of computation.
Using this smoothness
property, Siegel and Xu investigated
the error estimates of the approximation
\cite{CaOc20}.
Mhaskar and
Micchelli worked in compact sets
in ${\mathbb R}^n$,
while in the present work,
we consider the approximation on the whole real line.

A problem arises
when we deal with $k$-ReLU as a function
over the whole line.
The function $k$-ReLU is not bounded on ${\mathbb R}$.
Our goal in this paper is to propose a Banach space
that allows us to handle such unbounded functions.
Actually, 
for $k=1,2,\ldots$,
we let
\[
{\mathcal Y}_k({\mathbb R})
:=
\left\{f\in{\rm C}({\mathbb R})\,:\,
\lim_{x \to \pm\infty}\frac{f(x)}{1+|x|^k}
\mbox{ exists }\right\}
\]
equipped with the norm
\[
\|f\|_{{\mathcal Y}_k}
:=
\sup_{x \in {\mathbb R}}\frac{|f(x)|}{1+|x|^k}.
\]
and define
\[
H_{\text{ReLU}^k}({\mathbb R})
:=
{\rm Span}\left(
\left\{
\text{$k$-ReLU}(a \cdot+b)\,:\,a \ne 0, b \in {\mathbb R}
\right\}\right).
\]
Note that any element in ${\mathcal Y}_k$, divided by $1+|\cdot|^k$,
is a continuous function over 
$\overline{\mathbb R}:={\mathbb R} \cup \{\pm \infty\}$.
Our main result in this paper is as follows:
\begin{theorem}\label{thm:210111-1}
The linear subspace
$H_{{\rm ReLU}^k}({\mathbb R})$ is dense in ${\mathcal Y}_k({\mathbb R})$.
\end{theorem}
As is seen from the definition of the norm
$\|\cdot\|_{{\mathcal Y}_k}$,
when we have a function $f \in {\mathcal Y}_k$,
with ease,
we can find a function
$g \in H_{{\rm ReLU}^k}({\mathbb R})$
such that
$\lim\limits_{x \to \pm \infty}\frac{|f(x)-g(x)|}{1+|x|^k}=0$.
However, after choosing such a function $g$,
we have to look for a way to control
$f-g$ inside any compact interval by a function
$h \in {\mathcal Y}_k \cap C_{\rm c}({\mathbb R})$.
Although ${\mathcal Y}_k$ consists of unbounded functions,
we can manage to do so by induction on $k$.
Actually, we will find 
$h \in {\mathcal Y}_k \cap C_{\rm c}({\mathbb R})$
such that
$f-g-h$ is sufficiently small once we are given
a compact interval.

Theorem \ref{thm:210111-1} says that
the space
${\mathcal Y}_k({\mathbb R})$ is mathematically suitable
when we consider the activation function
$k$-ReLU.
We compare
Theorem \ref{thm:210111-1}
with the following fundamental result by Cybenko.
For a function space
$X({\mathbb R})$ over the real line ${\mathbb R}$
and an open set $\Omega$,
$X(\Omega)$ stands for
the restriction of each element $f$ to $\Omega$,
that is,
\[
X(\Omega)=\{f|\Omega\,:\,f \in X({\mathbb R})\}
\]
and the norm is given by
\[
\|f\|_{X(\Omega)}=\inf\{\|g\|_X\,:\,g \in X({\mathbb R}), \, f=g|\Omega\}.
\]
\begin{theorem}[Cybenko, \cite{Cybenko89}]\label{thm:Cybenko89-2}
Let $K\subset\Omega$ be a compact set, and $\sigma:{\mathbb R}\to{\mathbb R}$ be a continuous sigmoidal function.
Then, for all $f\in C(K)$ and $\varepsilon>0$, there exists 
$g\in H_\sigma(\Omega)$ such that
\begin{equation*}
\sup_{x\in K}|g(x)-f(x)|<\varepsilon.
\end{equation*}
\end{theorem}
We remark that 
Theorem \ref{thm:210111-1} is not a direct consequence
of
Theorem \ref{thm:Cybenko89-2}.
Theorem \ref{thm:Cybenko89-2}
concerns the 
uniform approximation over compact intervals,
while 
Theorem \ref{thm:Cybenko89-2} deals
with the uniform approximation over the whole real line.
We will prove
Theorem \ref{thm:210111-1}
without using
Theorem \ref{thm:Cybenko89-2}.

Let $k=0,1,\ldots$.
Our results readily can be carried over to
the case of $k$-sigmoidal functions.
As in
\cite[Definition 4.1]{Mhaskar92},
a continuous function $\sigma:{\mathbb R} \to {\mathbb R}$
is $k$-sigmoidal if 
\[
\lim_{x \to -\infty}\frac{\sigma(x)}{x^k}=0, \quad
\lim_{x \to \infty}\frac{\sigma(x)}{x^k}=1.
\]
Needless to say,
${\rm ReLU}^k$ is $k$-sigmoidal.
If $k=0$, then we say that $\sigma$ is a continuous sigmoidal.
As a corollary of Theorem \ref{thm:210111-1},
we extend this theorem to the case of
$k$-sigmoidal.
\begin{theorem}\label{thm:211102-1}
Let $X$ be a Banach lattice.
If $\sigma$ is $k$-sigmoidal,
then
the linear subspace
$H_{\sigma}({\mathbb R})$ is dense in ${\mathcal Y}_k({\mathbb R})$.
\end{theorem}

We can transplant
Theorem \ref{thm:210111-1} to various Banach lattices
over any open set $\Omega$ on the real line ${\mathbb R}$.
Here and below, 
$L^0(\Omega)$ denotes the set of all Lebesgue measurable functions from $\Omega$ to $\mathbb{C}$.
Let $X(\Omega)$ be a Banach space contained in $L^0(\Omega)$ endowed with the norm $\|\cdot\|_{X(\Omega)}$.
We say that $X(\Omega)$ is a Banach lattice if for any 
$f \in L^0(\Omega)$
and
$g \in {X(\Omega)}$ satisfying the estimate $|f(x)| \le |g(x)|$, a.e. $x\in \Omega$, 
$f \in X(\Omega)$ and
the estimate $\|f\|_{X(\Omega)} \le \|g\|_{X(\Omega)}$ holds.
See \cite{HIIS20}
for function spaces 
to which 
Theorem \ref{thm:210111-1} 
is applicable.

We write
\[
{\rm ReLU}^k_+(x)={\rm ReLU}^k(x)=\max(0,x)^k,
\quad
{\rm ReLU}_-^k(x)={\rm ReLU}^k_+(-x).
\]

\begin{theorem}[Universality on Banach lattices]\label{thm:201211-1}
Let $k \in {\mathbb N}$,
$\Omega\subset{\mathbb R}$ be an open set.
Assume
that ${\mathcal Y}_k(\Omega)\hookrightarrow X(\Omega)$.
Then 
$$\overline{H_{{\rm ReLU}^k}(\Omega)}^{\|\cdot\|_{X(\Omega)}}=
\overline{
{\mathbb C}({\rm ReLU}^k_+|\Omega)+{\mathbb C}({\rm ReLU}^k_-|\Omega)+C_{\rm c}({\mathbb R}^n)|\Omega
}^{\|\cdot\|_{X(\Omega)}}.
$$
\end{theorem}


It is noteworthy that
we can deal with the case
of $\Omega={\mathbb R}$.

\begin{remark}
The condition that 
$\chi_{\Omega} \in X(\Omega)$
is a natural condition,
since
$\sigma \in X(\Omega)$.
\end{remark}

\begin{remark}
Let $X(\Omega)$ be a Banach lattice , and 
let $\sigma$ be $1$-sigmoidal.
We put
\begin{equation*}
\sigma_0(x)
\equiv
{\rm ReLU}(x)-{\rm ReLU}(x-1),
\quad x\in{\mathbb R}.
\end{equation*}
Then, by the result for the case of $k=1$,
\begin{align*}
\overline{H_{\sigma_0}(\Omega)}^{\|\cdot\|_{X(\Omega)}}
=
\overline{H_\sigma(\Omega)}^{\|\cdot\|_{X(\Omega)}}.
\end{align*}
\end{remark}

\section{Proof of Theorem \ref{thm:210111-1}}

We need the following lemmas:
We embed $\mathcal{Y}_k$ into a function space
over $\overline{\mathbb R}={\mathbb R} \cup \{\pm\infty\}$.
\begin{lemma}\label{lem:210616-1}
The operator 
$\mathcal{Y}_k\rightarrow {\rm BC}(\overline{\mathbb{R}})$, 
$f\mapsto\dfrac f{1+|\cdot|^k}$ 
is an isomorphism.
\end{lemma}

If $k=1$,
then this can be found
in
\cite[Lemma 3]{HIIS20}.

\begin{proof}
Observe that the inverse
is given
for
$F \in {\rm BC}(\overline{\mathbb R})$ as follows:
\[
f(x)=(1+|x|)^k F(x)
\quad (x \in {\mathbb R}).
\]
Since the operator
$\mathcal{Y}_k\rightarrow {\rm BC}(\overline{\mathbb{R}})$, 
$f\mapsto\dfrac f{1+|\cdot|^k}$ 
preserves the norms,
we see that this operator
is an isomorphism.
\end{proof}
We set
\[
H_{{\rm ReLU}^k}^+({\mathbb R})
=
{\rm Span}
\{{\rm ReLU}^k(\cdot-t)\,:\,t \in {\mathbb R}\}.
\]We will use the following algebraic relation
for 
$H_{{\rm ReLU}^k}^+({\mathbb R})$.
\begin{lemma}\label{lem:210111-1}
Let $k\in{\mathbb N}$.
Then for all $x\in{\mathbb R}$,
\begin{equation*}
\sum_{j=0}^{k+1}\binom{k+1}j(-1)^j(x-j)^k=0,
\quad
\sum_{j=0}^k\binom kj(-1)^j(x-j)^k=
(-1)^k.
\end{equation*}
\end{lemma}

\begin{proof}[Proof of Lemma \ref{lem:210111-1}]
By comparing the coefficients, 
we may reduce the matter to the proof of
the following two equalities:
\begin{equation*}
\sum_{j=0}^k\binom{k+1}j(-1)^jj^\ell=0,
\quad
\sum_{j=0}^k\binom{k+1}j(-1)^jj^k
(-1)^k
\end{equation*}
for each $\ell=0,1\ldots,k-1$.
We compute
\begin{align*}
(e^t-1)^k
=
\sum_{j=1}^k\binom kj(-1)^{k-j}e^{jt},
\end{align*}
and then
\begin{align*}
t^k\left(\sum_{\ell=0}^\infty\frac{t^\ell}{(\ell+1)!}\right)^k
=
(-1)^k\sum_{\ell=0}^\infty\left[\sum_{j=1}^k\binom kj(-1)^jj^\ell\right]\frac{t^\ell}{\ell!}.
\end{align*}
Hence
\begin{equation*}
\sum_{j=0}^k\binom{k+1}j(-1)^jj^\ell=0,
\quad
\sum_{j=0}^k\binom{k+1}j(-1)^jj^k=(-1)^k
\end{equation*}
for each $\ell=0,1\ldots,k-1$.
\end{proof}

Although ${\rm ReLU}^k$
is unbounded,
if we consider suitable linear combinations,
we can approximate any function in
$C_{\rm c}({\mathbb R})$.
\begin{lemma}\label{lem:8}
Any function in
$C_{\rm c}({\mathbb R})$
can be approximated uniformly
over ${\mathbb R}$
by the functions
in 
$H_{{\rm ReLU}^k}({\mathbb R}) \cap
C_{\rm c}({\mathbb R})$.
More precisely,
if a function $f \in C_{\rm c}({\mathbb R})$
is contained in an interval $[-a',a]$
and $\varepsilon>0$,
then there exists
$\tau \in H_{{\rm ReLU}^k}({\mathbb R})$
such that
${\rm supp}(\tau) \subset [-a',a]$
and that
$\|\tau-f\|_{L^\infty} \le C \varepsilon$.
\end{lemma}

For the proof,
we will use the following observation:
If
\[
f=\sum_{j=1}^N
a_j {\rm ReLU}^{k-1}(\cdot-t_j),
\]
then,
by the definition of ${\rm ReLU}^k$,
\[
\int_{-\infty}^t
f(s)ds
=\sum_{j=1}^N
\frac{a_j}{k}
{\rm ReLU}^k(\cdot-t_j).
\]
\begin{proof}
We induct on $k$.
The base case $k=1$ was proved already
\cite{HIIS20}.
Suppose that we have
$f \in C_{\rm c}({\mathbb R})$
with ${\rm supp}(f) \subset [-a',a]$
for $a,a'>0$.
In fact, we can approximate $f$
with the functions in 
$H_{{\rm ReLU}^k}({\mathbb R})$
supported in the convex hull
$[-a',a]$
of ${\rm supp}(f)$.
Let $\varepsilon>0$ be given.
By mollification, we may assume
$f \in C^1({\mathbb R})$.
By the induction assumption,
there exists
$\psi \in H_{{\rm ReLU}^{k-1}}({\mathbb R})$
such that
\begin{equation}\label{eq:210609-1}
\|f'-\psi\|_{L^\infty}<(1+\ell)^{-1}\varepsilon,
\quad
{\rm supp}(\psi)
\subset
[-a',a]
\end{equation}
where
$\ell=a+a'={\rm diam}({\rm supp}(f))$.
Note that
\[
\varphi(t)=\int_{-\infty}^t \psi(s)ds
\quad (t \in {\mathbb R})
\]
is a function in
$H_{{\rm ReLU}^k}({\mathbb R})$.
Note that
\[
\varphi(t)=0
\mbox{ if }
t \le -a', \quad
\varphi(t)=
\int_{{\mathbb R}}\psi(s)ds=
\varphi(a)
\mbox{ if }
t \ge a.
\]
Integrating
estimate
(\ref{eq:210609-1}),
we obtain
\[
|f(t)-\varphi(t)|
\le
\int_{-a'}^t
|f'(s)-\varphi'(s)|ds
\le
\int_{-a'}^a
|f'(s)-\varphi'(s)|ds
\le
\frac{(a+a')\varepsilon}{1+\ell}
<\varepsilon
\]
for $t \ge -a'$.
In particular,
\[
\varphi(a)=\varphi(a)-f(a) \in
(-\varepsilon,\varepsilon).
\]
Thus,
$\|f-\varphi\|_{L^\infty}<\varepsilon$.
Using Lemma \ref{lem:210111-1},
the dilation and translation,
we choose
$\varphi^* \in H_{{\rm ReLU}^k}({\mathbb R})$,
which depends on $k$, $a$ and $a'$,
such that
${\rm supp}(\varphi^*) \subset (-a',\infty)$
and that
$\varphi^*$ agrees with $1$ over $(a,\infty)$.
If $t<-a'$,
then
for
$\tau=\varphi-\varphi(a)\varphi^*$,
\[
f(t)-\tau(t)
=
f(t)-\varphi(t)
\in
(-\varepsilon,\varepsilon).
\]
If $-a'<t<a$,
then
\[
f(t)-\tau(t)
=
f(t)-\varphi(t)
+\varphi(a)
\varphi^*(t)
\in
(-(1+\|\varphi^*\|_{L^\infty})\varepsilon,(1+\|\varphi^*\|_{L^\infty})\varepsilon).
\]
Finally,
if $t>a$,
then
\[
f(t)-\tau(t)
=
f(t)-\varphi(t)
+\varphi(a)
\varphi^*(t)
=
f(t)-\varphi(t)
+\varphi(a)=0.
\]
Therefore, the function
$\tau$ is 
a function
in
$H_{{\rm ReLU}^k}({\mathbb R})$
satisfying
${\rm supp}(\tau)
\subset[-a',a]$
and
$\|f-\tau\|_{L^\infty}<C\varepsilon$,
where $C$ depends on $k$,
$a$ and $a'$,
that is,
$k$ and $f$.
\end{proof}

We will prove Theorems \ref{thm:210111-1} and \ref{thm:211102-1}.
\begin{proof}[Proof of Theorem \ref{thm:210111-1}]
We identify ${\mathcal Y}_k$ with ${\rm BC}(\overline{\mathbb R})$
as in Lemma \ref{lem:210616-1}.
We have to show that any finite Borel measure $\mu$
in $\overline{\mathbb R}$
which annihilates
$H_{{\rm ReLU}^k}({\mathbb R})$
is zero.
Since $C_{\rm c}({\mathbb R})$ is contained in 
the closure of the space
$H_{{\rm ReLU}^k}({\mathbb R})$
as we have seen in Lemma \ref{lem:8},
$\mu$ is not supported on ${\mathbb R}$.
Therefore, we have only to show that
$\mu(\{\infty\})=0$
and that
$\mu(\{-\infty\})=0$.
However, 
since
we have
shown that $\mu$ is not supported on ${\mathbb R}$,
this is a direct consequence of the following observations:
\[
\mu(\{\infty\})
=
\int_{\overline{\mathbb R}}
\frac{{\rm ReLU}^k(t)}{1+|t|^k}d\mu(t)=0, \quad
\mu(\{-\infty\})=(-1)^k
\int_{\overline{\mathbb R}}
\frac{{\rm ReLU}^k(-t)}{1+|t|^k}d\mu(t)=0.
\]
Thus, $\mu=0$.
\end{proof}

\begin{proof}[Proof of Theorem \ref{thm:211102-1}]
We identify ${\mathcal Y}_k$ with ${\rm BC}(\overline{\mathbb R})$
as in Lemma \ref{lem:210616-1} once again.
Then
to show that
\[
H_\sigma({\mathbb R})=
{\rm Span}\left\{
\frac{\sigma(a x-b)}{1+|x|^k}\,:\,a,b \in {\mathbb R}
\right\}
\]
is dense in ${\rm BC}(\overline{\mathbb R})$
under this identification,
it suffices to show that
any finite measure $\mu$ over $\overline{\mathbb R}$
is zero
if it annihilates $H_\sigma({\mathbb R})$.

Assuming that $\mu$ annihilates $H_\sigma({\mathbb R})$,
we see that
\begin{equation}\label{eq:211212-1}
\int_{\overline{\mathbb R}}
\frac{\sigma(a x-a b)}{1+|x|^k}d\mu(x)=0
\end{equation}
for any $a \ne 0$ and $b \in {\mathbb R}$.
Since $\sigma$ is $k$-sigmoidal,
\[
\sup_{x \in {\mathbb R}^n}
\sup_{a \in {\mathbb R} \setminus [-1,1]}
\left|\frac{\sigma(a x-a b)}{a^k(1+|x|^k)}\right|<\infty
\]
for any fixed $b \in {\mathbb R}$.
Furthermore,
\[
\lim_{a \to \infty}\frac{\sigma(a x-a b)}{|a|^k(1+|x|^k)}=\frac{(x-b)_+^k}{1+|x|^k}, \quad
\lim_{a \to -\infty}\frac{\sigma(a x-a b)}{|a|^k(1+|x|^k)}=\frac{(b-x)_+^k}{1+|x|^k}.
\]
Therefore, by the Lebesgue convergence theorem,
letting $a \to \pm\infty$ in (\ref{eq:211212-1}),
we have
\[
\int_{\overline{\mathbb R}}
\frac{(x-b)_+^k}{1+|x|^k}d\mu(x)=
\int_{\overline{\mathbb R}}
\frac{(b-x)_+^k}{1+|x|^k}d\mu(x)=0.
\]
This means that
$\mu$ annihilates $H_{{\rm ReLU}^k}({\mathbb R})$.
Thus, by Theorem \ref{thm:210111-1},
$\mu=0$.
\end{proof}

\section{Proof of Theorem \ref{thm:201211-1}}

We show
\begin{equation}\label{eq:220115-1}
H_{{\rm ReLU}^k}
=
{\mathbb C}{\rm ReLU}^k_+
+
{\mathbb C}{\rm ReLU}^k_-
+
\overline{C_{\rm c}({\mathbb R})}^{{\mathcal Y}_k}.
\end{equation}
We have
\[
H_{{\rm ReLU}^k}
\supset
\overline{
C_{\rm c}({\mathbb R})
}^{{\mathcal Y}_k}
\]
by Lemma \ref{lem:8}.
Hence,
\[
H_{{\rm ReLU}^k}\supset
{\mathbb C}{\rm ReLU}^k_+
+
{\mathbb C}{\rm ReLU}^k_-
+
\overline{C_{\rm c}({\mathbb R})}^{{\mathcal Y}_k}.
\]
Thus, we prove the opposite inclusion.

For any $f \in H_{{\rm ReLU}^k}$,
there exist
$\beta_\pm \in {\mathbb C}$ such that
$g_0(x)=f(x)-\beta_+{\rm ReLU}^k_+(x)-\beta_-{\rm ReLU}^k_-(x)$
is a polynomial of degree $(k-1)$ both on $[K,\infty)$ and
on
$(-\infty,-K]$
for $K \gg 1$.
Fix $R \gg K$ for the time being.
Then we have
\[
g(x)=\beta_+{\rm ReLU}^k_+(x)+\beta_-{\rm ReLU}^k_-(x)
\]
satisfies
\[
\sup_{x \in {\mathbb R} \setminus [-R,R]}
\frac{|g_0(x)|}{1+|x|^k}=
\sup_{x \in {\mathbb R} \setminus [-R,R]}
\frac{|f(x)-g(x)|}{1+|x|^k}={\rm O}(R^{-1}).
\]
We define
\[
F(x)
=
\begin{cases}
-g_0(R)(x-R)+g_0(R)&R \le x \le R+1,\\
g(x)&|x| \le R,\\
g_0(-R)(x+R)+g_0(-R)&-R-1 \le x \le -R,\\
0&\mbox{otherwise}.
\end{cases}
\]
By the use of Lemma \ref{lem:8},
we choose a compactly supported function
$$
h \in H_{{\rm ReLU}^k}({\mathbb R}) \cap C_{\rm c}({\mathbb R})
$$
supported on $[-R-2,R+2]$ so that
\[
\sup_{x \in {\mathbb R}}|F(x)-h(x)|
=
\sup_{x \in [-R-2,R+2]}|F(x)-h(x)|
\le R^{-1}.
\]
Then we have
\[
\sup_{x \in [-R,R]}
\frac{|f(x)-g(x)-h(x)|}{1+|x|^k} 
\le CR^{-1}, \quad
\sup_{x \in [-R-2,R+2]}|f(x)-g(x)-h(x)| \le C\,R^{k-1}
\]
Then we have
\begin{align*}
\|g_0-h\|_{{\mathcal Y}_k}&=
\sup_{x \in {\mathbb R}}
\frac{|f(x)-g(x)-h(x)|}{1+|x|^k}\\
&\le
\sup_{x \in [-R,R]}
\frac{|f(x)-g(x)-h(x)|}{1+|x|^k}+
\sup_{x \in {\mathbb R} \setminus [-R-2,R+2]}
\frac{|f(x)-g(x)|}{1+|x|^k}\\
&\quad+
\sup_{x \in [-R-2,R+2] \setminus [-R,R]}
\frac{|f(x)-g(x)-h(x)|}{1+|x|^k}\\
&\le
\sup_{x \in [-R,R]}
\frac{|f(x)-g(x)-h(x)|}{1+|x|^k}+
\sup_{x \in {\mathbb R} \setminus [-R-2,R+2]}
\frac{|f(x)-g(x)|}{1+|x|^k}\\
&\quad+C
\sup_{x \in [-R-2,R+2] \setminus [-R,R]}
\frac{R^{k-1}}{1+|x|^k}\\
&={\rm O}(R^{-1}).
\end{align*}
Since
$g \in {\mathbb C}{\rm ReLU}_+^k
+{\mathbb C}{\rm ReLU}_-^k$,
$h \in C_{\rm c}({\mathbb R})$
and
$\|f-h-g\|_{{\mathcal Y}_k}=\|g_0-h\|_{{\mathcal Y}_k}<\varepsilon$
as long as $R$ is large enough,
$f-g \in \overline{C_{\rm c}({\mathbb R})}^{{\mathcal Y}_k}$.
Thus, we obtain
(\ref{eq:220115-1}).

From
(\ref{eq:220115-1}),
we deduce
\begin{align*}
\overline{H_{{\rm ReLU}^k}(\Omega)}^{{\mathcal Y}_k}
&\subset
\overline{
{\mathbb C}({\rm ReLU}^k_+|\Omega)
+
{\mathbb C}({\rm ReLU}^k_-|\Omega)
+
\overline{
C_{\rm c}({\mathbb R})|\Omega}^{{\mathcal Y}_k}}^{{\mathcal Y}_k}\\
&=
\overline{
{\mathbb C}({\rm ReLU}^k_+|\Omega)
+
{\mathbb C}({\rm ReLU}^k_-|\Omega)
+
C_{\rm c}({\mathbb R})|\Omega}^{{\mathcal Y}_k}.
\end{align*}

Thus, the proof is complete
if $X={\mathcal Y}_k$.
For general Banach lattices
$X$,
we use a routine approximation procedure.
We prove
$$\overline{H_{{\rm ReLU}^k}(\Omega)}^{\|\cdot\|_{X(\Omega)}}=
\overline{
{\mathbb C}({\rm ReLU}^k_+|\Omega)+{\mathbb C}({\rm ReLU}^k_-|\Omega)+C_{\rm c}({\mathbb R}^n)|\Omega
}^{\|\cdot\|_{X(\Omega)}}.
$$
Let
$f \in \overline{H_{{\rm ReLU}^k}(\Omega)}^{\|\cdot\|_{X(\Omega)}}$
and
$\varepsilon>0$.
Then 
since
$f \in \overline{H_{{\rm ReLU}^k}(\Omega)}^{\|\cdot\|_{X(\Omega)}}$
there exists
$f_0 \in H_{{\rm ReLU}^k}$
such that
\begin{equation}\label{eq:220115-2}
\|f-f_0|\Omega\|_{X(\Omega)}<\varepsilon.
\end{equation}
Since we know that
\[
f_0 \in
H_{{\rm ReLU}^k}
\subset
{\mathbb C}{\rm ReLU}^k_+
+
{\mathbb C}{\rm ReLU}^k_-
+
\overline{C_{\rm c}}^{{\mathcal Y}_k},
\]
there exist constants $\beta_\pm$ and $h \in C_{\rm c}({\mathbb R})$
such that
$\|f_0-\beta_+{\rm ReLU}^k_+-\beta_-{\rm ReLU}^k_--h\|_{{\mathcal Y}_k}
<\varepsilon$.
Hence for such $\beta_\pm$ and $h \in C_{\rm c}({\mathbb R})$,
we have
$\|f_0|\Omega-\beta_+{\rm ReLU}^k_+|\Omega
-\beta_-{\rm ReLU}^k_-|\Omega-h|\Omega\|_{{\mathcal Y}_k(\Omega)}
<\varepsilon$.
Since we assume that
${\mathcal Y}_k(\Omega)\hookrightarrow X(\Omega)$,
we have
$\|f-\beta_+{\rm ReLU}^k_+|\Omega-\beta_-{\rm ReLU}^k_-|\Omega-h|\Omega\|_{X(\Omega)}
<\varepsilon$.
Therefore,
we have
\[
\overline{H_{{\rm ReLU}^k}(\Omega)}^{\|\cdot\|_{X(\Omega)}}
\subset
\overline{
{\mathbb C}({\rm ReLU}^k_+|\Omega)+{\mathbb C}({\rm ReLU}^k_-|\Omega)+C_{\rm c}({\mathbb R}^n)|\Omega
}^{\|\cdot\|_{X(\Omega)}}.
\]

\vspace{10pt}
{\em Availability of data and materials}.
No data and materials were used to support this study.

\vspace{10pt}
{\em Competing interests}.
The authors declare that there are no conflicts of interest regarding the publication of this paper.

\vspace{10pt}
{\em Funding}.
This work was supported by a JST CREST Grant (Number JPMJCR1913, Japan).
This work was also supported by the RIKEN Junior Research Associate Program.
The second author 
was
supported by a Grant-in-Aid for Young Scientists Research 
(No.19K14581), Japan Society for the Promotion of Science.
The fourth author 
was
supported by a 
Grant-in-Aid for Scientific Research (C) (19K03546), 
Japan Society for the Promotion of Science.

\vspace{10pt}
{\em Authors' contributions}.
The four authors contributed equally to this paper.
All of them read the whole manuscript and approved the content of the paper.

\end{document}